\documentclass[11pt]{article}
\usepackage{graphicx} 
\usepackage[a4paper,margin=2.5cm]{geometry}
\usepackage{amsmath,amssymb,amsthm,mathtools}
\usepackage{booktabs}
\usepackage{array}
\usepackage{enumitem}
\usepackage[colorlinks=true,linkcolor=blue,citecolor=blue,urlcolor=blue]{hyperref}

\newtheorem{theorem}{Theorem}[section]
\newtheorem{lemma}[theorem]{Lemma}
\newtheorem{proposition}[theorem]{Proposition}
\newtheorem{corollary}[theorem]{Corollary}

\theoremstyle{definition}
\newtheorem{definition}[theorem]{Definition}
\newtheorem{example}[theorem]{Examples}
\theoremstyle{remark}
\newtheorem{remark}[theorem]{Remark}
\newtheorem*{mainthmrestate}{Theorem~\ref{thm-main}}
\numberwithin{equation}{section}

\newcommand{\En}{\mathcal{E}}

\newcommand{\Rook}{\operatorname{Rook}}
\newcommand{\tr}{\operatorname{tr}}
\newcommand{\sgn}{\varsigma}

\begin{document}

  \begin{center}
 {\Large \bf Energy of signed graphs and the vertex cover number}\\

 \vspace{6mm}

 \bf {Aniruddha Samanta$^a$, Deepshikha$^b$, Kinkar Chandra Das$^{c,*}$}

 \vspace{5mm}

{\it $^a$Department of Mathematics, Birla Institute of Technology,\\
Mesra, Ranchi-835215, India\/} \\
{\tt E-mail: aniruddha.sam@gmail.com,\,aniruddha@bitmesra.ac.in}\\[2mm]

{\it $^b$Department of Mathematics, Shyampur Siddheswari Mahavidyalaya, University of Calcutta,\\
 West Bengal 711312, India\/} \\
{\tt E-mail: dpmmehra@gmail.com,\, deepshikha@ssmmahavidyalaya.edu.in}\\[2mm]

{\it $^c$Department of Mathematics, Sungkyunkwan University, \\
 Suwon 16419, Republic of Korea\/} \\
{\tt E-mail: kinkardas2003@gmail.com,\,kinkar@skku.edu}

 \hspace*{30mm}


 \end{center}

 \baselineskip=0.23in


\begin{abstract}
The energy $\mathcal{E}(\Gamma)$ of a signed graph $\Gamma=(G,\sigma)$ is the sum of the absolute values of the eigenvalues of its adjacency matrix $A(\Gamma)$. In this article, we prove that $\mathcal{E}(\Gamma)\geq 2\tau(G)$, where $\tau(G)$ is the vertex cover number of $G$. The case of unsigned graphs, a 1980s conjecture proved only recently, is a special case of this result. Moreover, we completely characterize the equality case: equality holds if and only if every connected component of $\Gamma$ is switching equivalent to an \emph{additive rook graph} $\Rook_N(a,b)$. This graph is constructed on an $N\times N$ board with $a_i+b_j$ vertices on square $(i,j)$. Two vertices are joined by a positive edge if they share only a row, and by a negative edge if they share only a column. As consequences, we obtain a balanced/antibalanced trichotomy for the extremal signed graphs and deduce the characterization of the extremal unsigned graphs. Finally, an example shows that the bound does not extend to the Hermitian adjacency matrices of mixed graphs or to complex unit gain graphs.
\end{abstract}



 \baselineskip=0.25in

\section*{Introduction}

Throughout this article, $G$ denotes a simple undirected graph with vertex set
$V(G)$ (or simply $V$) and edge set $E(G)$. If two vertices $u$ and $v$ are adjacent, we write $u\sim v$ and denote the edge between them by $uv$. Let
$\lambda_1\geq\lambda_2\geq\cdots\geq\lambda_n$ be the eigenvalues of the adjacency matrix
$A(G)$. The \emph{energy} of $G$ is
\[
\mathcal{E}(G)=\sum_{i=1}^{n}|\lambda_i| .
\]
Gutman introduced graph energy in 1978 \cite{Gutman1978}, motivated by its connection with the total $\pi$-electron energy of conjugated hydrocarbons. Since then it has been studied extensively; see \cite{LiShiGutman, AkbariSeidel, DaySo, SamantaKannanExtremal,DasMojallalGutmanMcClelland, DasMojallalUpper, DasMojallalSun} and the references therein. Because the energy depends on the whole spectrum, it can rarely be computed in closed form. Much of the literature is therefore devoted to bounding it in terms of simpler graph parameters, such as the number of vertices or edges, the chromatic number, the matching number and the vertex cover number, see \cite{AkbariAlazemiAndelic, Nikiforov, XuZhouWongTian, HeHaoDong}\\[-3mm]

A set $S\subseteq V(G)$ is a \emph{vertex cover} of $G$ if every edge of $G$ has at least one end in $S$. The \emph{vertex cover number} $\tau(G)$ is the size of a minimum vertex cover. A set of vertices is \emph{independent} if no two of its vertices are adjacent, and the \emph{independence number} $\alpha(G)$ is the size of a maximum independent set. The complement of a vertex cover is an independent set and vice versa, so $\tau(G)+\alpha(G)=n$.\\[-3mm]

In the 1980s, Fajtlowicz~\cite{Fajtlowicz,DeLaVina} conjectured, on the basis of his computer program Graffiti, that every graph $G$ satisfies
\begin{equation}\label{eq-bound-unsigned}
\mathcal{E}(G)\geq 2\tau(G);
\end{equation} see also \cite[Conjecture~543]{AouchicheHansen} and \cite{LiuNing}. The conjecture attracted considerable attention. Wang and Ma~\cite{WangMa} proved the weaker bound $\mathcal{E}(G)\geq2\tau(G)-2c(G)$, where $c(G)$ is the number of odd cycles of $G$, and Samanta and Kannan~\cite{GainWM} extended it to complex unit gain graphs. Wong, Wang and Chu~\cite{WongWangChu} proved $\mathcal{E}(G)\geq2\mu(G)$ in terms of the matching number $\mu(G)$, with analogues for oriented graphs~\cite{TianWongSkew}, mixed graphs~\cite{WeiLi} and complex unit gain graphs~\cite{LiYuxuan}. Since $\tau(G)\geq\mu(G)$, inequality \eqref{eq-bound-unsigned} improves this last bound. Inequality \eqref{eq-bound-unsigned} was verified for several classes of graphs in \cite{AkbariKSY,OurPaper}, and Abiad, Coutinho, Juliano and Reijnders~\cite{ACJR} obtained further partial results using semidefinite programming. The conjecture was finally settled by Kumar and Pragada~\cite{KumarPragada}, who proved that \eqref{eq-bound-unsigned} holds for every graph. Shortly afterwards, Mojallal~\cite{Mojallal} determined all graphs attaining equality.\\[-3mm]

It is natural to ask whether \eqref{eq-bound-unsigned} survives when the edges carry weights. A \emph{signed graph} \cite{ZaslavskySigned} $\Gamma=(G,\sigma)$ consists of an underlying graph $G$ together with a \emph{signature} $\sigma:E(G)\to\{-1,+1\}$. Its adjacency matrix $A(\Gamma)$ is a symmetric $|V|\times|V|$ matrix, defined by
\[
(u,v)\text{-th entry of $A(\Gamma)$}=A(\Gamma)_{uv}=
\begin{cases}
	\sigma(uv), & \text{if } u\sim v,\\
	0, & \text{otherwise,}
\end{cases}
\]
and its energy $\mathcal{E}(\Gamma)$ is the sum of the absolute values of the eigenvalues of $A(\Gamma)$. An unsigned graph is the special case in which every edge is positive, that is, $A(G)=A((G,+1))$. Signed graphs are in turn a special case of the complex unit gain graphs ($\mathbb{T}$-gain graphs) introduced by Reff \cite{Reff}, whose edges carry complex numbers of modulus one. \emph{Switching} at a vertex set $U$ reverses the sign of every edge with exactly one end in $U$; it does not change the spectrum. A signed graph is \emph{balanced} if it can be switched to the all-positive signature, and \emph{antibalanced} if $(G,-\sigma)$ is balanced.\\[-3mm]

In this article, we prove that \eqref{eq-bound-unsigned} holds for every signed graph and we completely characterize the case of equality. The extremal signed graphs are described by the following family. Let $N\geq 1$ and let $a=(a_1,\ldots,a_N)$ and $b (b_1,\ldots,b_N)$ be vectors of nonnegative integers. The \emph{additive rook graph} $\Rook_N(a,b)$ is the signed graph built on an $N\times N$ board with $a_i+b_j$ vertices on square $(i,j)$, in which two vertices are joined by a positive edge if they lie in the same row but in different columns, and by a negative edge if they lie in the same column but in different rows. All other pairs of vertices are non-adjacent. Our main result is the following.

\begin{theorem}\label{thm-main}
	Let $\Gamma=(G, \sigma)$ be a signed graph with vertex cover number $\tau(G)$. Then
	\begin{equation}\label{eq-bound-signed}
		\mathcal{E}(\Gamma)\geq 2\tau(G).
	\end{equation}
	Equality holds if and only if every connected component of $\Gamma$ is switching equivalent to
	an additive rook graph $\Rook_N(a,b)$.
\end{theorem}

Theorem~\ref{thm-main} reveals features that do not appear in the unsigned case. The class of underlying graphs that attain equality is strictly larger than for unsigned graphs, and equality is attained by some signed graphs that are neither balanced nor antibalanced. This leads to a balanced/antibalanced trichotomy for the extremal signed graphs. Restricting Theorem~\ref{thm-main} to unsigned graphs, we recover Mojallal's characterization \cite{Mojallal}. On the other hand, the bound does not extend further: we give an example showing that \eqref{eq-bound-signed} fails for the Hermitian adjacency matrices of mixed graphs and for $\mathbb{T}$-gain graphs.\\[-3mm]

The article is organized as follows. Section~\ref{sec-prelim} collects the notation, definitions and preliminary results used throughout. In Section~\ref{sec-rook}, we introduce additive rook graphs and establish their basic properties. Section~\ref{sec-bound} contains the proof of \eqref{eq-bound-signed}. Section~\ref{sec-equality} characterizes equality in two steps: we first show that every additive rook graph attains equality, and then prove the converse, namely that every signed graph attaining equality is switching equivalent to a disjoint union of additive rook graphs. Finally, in Section~\ref{sec-consequences}, we state Mojallal's theorem in full and deduce it from Theorem~\ref{thm-main}, and we present an example showing that the bound cannot be extended to Hermitian adjacency matrices of mixed graphs or to $\mathbb{T}$-gain graphs.\\

\section{Preliminaries}\label{sec-prelim}

Throughout this article, $G$ denotes a finite simple undirected graph with vertex set $V(G)$ and edge set $E(G)$. We put $n=|V(G)|$ and $m=|E(G)|$. The edge joining two vertices $u$ and $v$ is written $uv$, and we write $u\sim v$ when $uv\in E(G)$. For a vertex $v$, the set $N(v)=\{u\in V(G): u\sim v\}$ is its neighbourhood, $N[v]=N(v)\cup\{v\}$ is its closed neighbourhood, and $d_v=|N(v)|$ is its degree. For a signed graph $\Gamma=(G,\sigma)$, every purely graph-theoretic notion is taken from the underlying graph $G$. So the vertex covers and the independent sets of $\Gamma$ are those of $G$, and we write $\tau(\Gamma)=\tau(G)$ and $\alpha(\Gamma)=\alpha(G)$. For $U\subseteq V(G)$, the signed graph induced by $\Gamma$ on $V(G)\setminus U$ is denoted by $\Gamma-U$, and its edges keep their signs from $\Gamma$.\\[-3mm]

If the rows and columns of  a matrix $M$ are indexed by $V(G)$, then the $(u,v)$-th entry of $M$ is written $M_{uv}$. The adjacency matrix $A(\Gamma)$ is the real symmetric matrix with $A(\Gamma)_{uv}=\sigma(uv)$ if $u\sim v$, and $A(\Gamma)_{uv}=0$ otherwise. If $\lambda_1\geq\cdots\geq\lambda_n$ are the eigenvalues of $A(\Gamma)$, the energy of $\Gamma$ is defined as $\mathcal{E}(\Gamma)=\sum_{k=1}^{n}|\lambda_k|$.\\[-3mm]

A \emph{switching function} is a map $\zeta:V(G)\to\{1,-1\}$. It turns $\Gamma=(G,\sigma)$ into the signed graph $\Gamma^{\zeta}=(G,\sigma^{\zeta})$, where $\sigma^{\zeta}(uv)=\zeta(u)\,\sigma(uv)\,\zeta(v)$ for every edge $uv$. We then say that $\Gamma$ and $\Gamma^{\zeta}$ are \emph{switching equivalent}, and write $\Gamma\sim\Gamma^{\zeta}$. With $D=\operatorname{diag}\bigl(\zeta(v)\bigr)_{v\in V(G)}$ we have $A(\Gamma^{\zeta})=D\,A(\Gamma)\,D$, and we also write $\Gamma^{D}$ for $\Gamma^{\zeta}$. As $D^{-1}=D$, switching equivalent signed graphs have the same eigenvalues, and so the same energy. The sign of a cycle is the product of the signs of its edges, and the cycle is \emph{positive} when this product equals $1$. A signed graph is \emph{balanced} if all of its cycles are positive. The negation of $\Gamma$ is $-\Gamma=(G,-\sigma)$. Since $A(-\Gamma)=-A(\Gamma)$, we get $\mathcal{E}(-\Gamma)=\mathcal{E}(\Gamma)$. We call $\Gamma$ \emph{antibalanced} if $-\Gamma$ is balanced.\\[-3mm]

For a real symmetric matrix $S$, we write $S\succcurlyeq0$ if $S$ is positive semidefinite, and $S\succcurlyeq T$ if $S-T\succcurlyeq0$. We use $J$ for the all-ones matrix and $\mathbf{1}$ for the all-ones vector, and we write $\mathbf{1}_k$ when its length $k$ has to be shown. For other notation and basic facts from matrix theory, we refer to~\cite{HornJohnson}. In addition, we shall repeatedly use the following standard facts.\\[-3mm]

\begin{itemize}
  \item[\textbf{(F1)}] Every real symmetric $n\times n$ matrix $S$ can be written as
  $S=\sum_{k=1}^{n}\lambda_k u_ku_k^{\top}$, where $\lambda_1,\ldots,\lambda_n$ are the
  eigenvalues of $S$ and $u_1,\ldots,u_n$ is an orthonormal basis of corresponding eigenvectors.

  \item[\textbf{(F2)}] For $S$ as in (F1), the matrices $S_{+}=\sum_{\lambda_k>0}\lambda_k u_ku_k^{\top}$ and $S_{-}=-\sum_{\lambda_k<0}\lambda_k u_ku_k^{\top}$ satisfy $S_{\pm}\succcurlyeq0$,
  $S_{+}S_{-}=S_{-}S_{+}=0$ and $S=S_{+}-S_{-}$. We put $|S|=S_{+}+S_{-}$. Note that $|S|$ is not the matrix of the absolute values of the entries of $S$. The energy of $S$ is $\mathcal{E}(S)=\operatorname{tr}|S|=\sum_k|\lambda_k|$. If in addition $\operatorname{tr}S=0$, then $\operatorname{tr}S_{+}=\operatorname{tr}S_{-}=\tfrac12\mathcal{E}(S)$.

  \item[\textbf{(F3)}] A matrix $S$ with rows and columns indexed by $V$ satisfies
  $S\succcurlyeq0$ if and only if there are vectors $s_v$ $(v\in V)$ in some $\mathbb{R}^k$ such that $S_{uv}=\langle s_u,s_v\rangle$ for all $u,v\in V$. These vectors are called a \emph{Gram representation} of $S$.

  \item[\textbf{(F4)}] For $x,y\in\mathbb{R}^k$ we have $|\langle x,y\rangle|\leq\|x\|\,\|y\|$, and equality holds if and only if $x$ and $y$ are linearly dependent. In particular, if $\|x\|=\|y\|=r>0$, then $\langle x,y\rangle=r^2$ exactly when $x=y$, and $\langle x,y\rangle=-r^2$ exactly when $x=-y$.
\end{itemize}

Finally, we recall two families of graphs. The complete multipartite graph $K_{x,\ldots,x}$ with $r$ parts consists of $r$ independent sets of size $x$, and two vertices are adjacent exactly when they lie in different parts. For $r\geq2$ and $x,y\geq1$, the graph $H_r(x,y)$ is built from two complete multipartite graphs, one with parts $X_1,\ldots,X_r$ of size $x$ and the other with parts $Y_1,\ldots,Y_r$ of size $y$, by joining every vertex of $X_i$ to every vertex of $Y_i$ for each $i$. Thus the edges of $H_r(x,y)$ are exactly the pairs $uv$ with $u\in X_i$, $v\in X_j$ and $i\neq j$, the pairs $uv$ with $u\in Y_i$, $v\in Y_j$ and $i\neq j$, and the pairs $uv$ with
$u\in X_i$ and $v\in Y_i$.

\section{Additive rook graphs and independence numbers }\label{sec-rook}

In this section we introduce \emph{additive rook graphs}. These signed graphs turn out to be exactly the extremal graphs for the energy bound \eqref{eq-bound-signed}. We first define a more general family called \emph{signed grid graphs} and obtain additive rook graphs as a special case. We then determine the independence number and the vertex cover number of additive rook graphs. Finally, we prove a few further properties that are needed in Section~\ref{sec-equality}.

\begin{definition}[Signed grid graph]\label{def-grid}
	Let $W=(w_{ij})$ be an $s\times t$ matrix of nonnegative integers. The \emph{signed grid graph} $\Gamma[W]$ has vertex set $V:=\bigsqcup_{i,j}V_{ij}$ where $|V_{ij}|=w_{ij}$. The set $V_{ij}$ is called the \emph{square} $(i,j)$. The $i$-th \emph{row} is $R_i:=\bigcup_{j=1}^{t}V_{ij}$ and the $j$-th \emph{column} is $C_j:=\bigcup_{i=1}^{s}V_{ij}$. For $v\in V_{ij}$ we write $r(v)=i$ and $c(v)=j$, meaning that $v$ lies in the $i$-th row and the $j$-th column. Two distinct vertices $u$ and $v$ of $\Gamma[W]$ are joined by a positive edge if $r(u)=r(v)$ and $c(u)\neq c(v)$, and by a negative edge if $c(u)=c(v)$ and $r(u)\neq r(v)$. All other pairs of distinct vertices are non-adjacent.
\end{definition}

Next we define two matrices and a graph.

\begin{definition}\label{def-incidence}
	Let $\Gamma[W]$ be the signed grid graph of an $s\times t$ matrix $W$. We define the following matrices and graph associated with $\Gamma[W]$.
	\begin{itemize}
		\item[(i)] The \emph{row-incidence matrix} $R$ is the $|V|\times s$ matrix with $R_{vi}=1$ if $r(v)=i$ and $R_{vi}=0$ otherwise.
        
		\item[(ii)] The \emph{column-incidence matrix} $C$ is the $|V|\times t$ matrix with $C_{vj}=1$ if $c(v)=j$ and $C_{vj}=0$ otherwise.
        
		\item[(iii)] The \emph{cell graph} $\mathcal{B}(W)$ is the bipartite graph whose vertices are the nonempty rows and the nonempty columns, in which row $i$ is adjacent to column $j$ if and only if $w_{ij}>0$.
	\end{itemize}
\end{definition}

\begin{example}
	Let $W=\begin{bmatrix}
		 3 & 2 & 3 \\
		2 & 3 & 2 
	\end{bmatrix}$ . The signed grid graph $\Gamma[W]$ is shown in Figure ~\ref{fig-grid}. In $\Gamma[W]$, the vertex set $V=\cup_{i=1}^{2}\cup_{j=1}^{3}V_{ij}$. It has two rows, $R_1=V_{11}\cup V_{12} \cup V_{13}$ and $R_2=V_{21}\cup V_{22} \cup V_{23}$. Also, it has three columns $C_1=V_{11}\cup V_{21}$, $C_2=V_{12}\cup V_{22}$, and $C_3=V_{13}\cup V_{23}$. 
\end{example}

\begin{figure}
	\begin{center}
		\includegraphics[scale= 0.37]{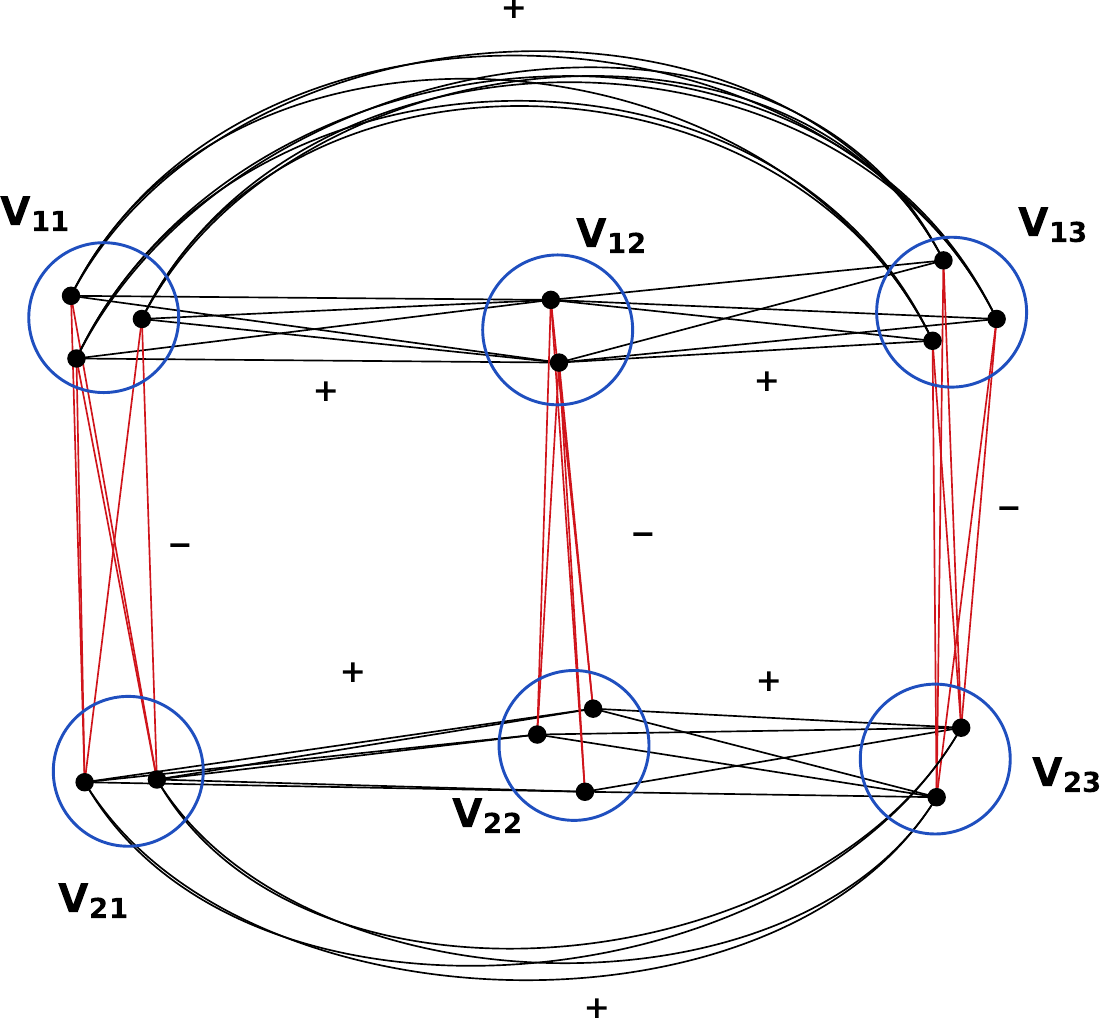}
		\caption{The signed grid graph $\Gamma[W]$.}\label{fig-grid}
	\end{center}
\end{figure}

\begin{definition}[Additive rook graph]\label{def-rook}
	Let $N\geq1$ and let $a=(a_1,\ldots,a_N)$ and $b=(b_1,\ldots,b_N)$ be nonnegative integer vectors, not both zero. The \emph{additive rook graph} $\Rook_N(a,b)$ is the signed grid graph $\Gamma[W]$, where $W$ is the $N\times N$ matrix with entries $w_{ij}=a_i+b_j$. We denote $\ell:=\sum_{i=1}^{N}a_i+\sum_{j=1}^{N}b_j .$
\end{definition}

Thus $\Rook_N(a,b)$ is obtained by placing $a_i+b_j$ vertices on square $(i,j)$ of an $N\times N$ chessboard; vertices sharing only a row are joined by positive edges, and vertices sharing only a column by negative edges.

\begin{example}\label{ex-members}
	The following signed graphs are additive rook graphs up to switching equivalence. Here $\mathbf{0}$ is the zero vector of suitable size.
	\begin{center}
		\begin{tabular}{@{}ll@{}}
			\toprule
			Signed graph & As $\Rook_N(a,b)$\\
			\midrule
			$K_1$ & $\Rook_1\bigl((1),(0)\bigr)$\\
			$K_{x,\ldots,x}$ with $r$ parts & $\Rook_r\bigl((x,0,\ldots,0),\mathbf{0}\bigr)$\\
			$H_r(x,y)$ & $\Rook_r\bigl((x,y,0,\ldots,0),\mathbf{0}\bigr)$\\
			$-K_{x,\ldots,x}$ and $-H_r(x,y)$ & interchange $a$ and $b$ above\\
			$K_N\,\square\,K_N$, rows positive, columns negative & $\Rook_N\bigl((1,\ldots,1),\mathbf{0}\bigr)$\\
			positive $K_{1,2,1}$ and negative $K_{2,1,1}$ glued along the part of size $2$ & $\Rook_3\bigl((1,0,0),(0,1,0)\bigr)$(\textbf{See})\\
			\bottomrule
		\end{tabular}
	\end{center}
    
Consider the third line of the table. In $\Rook_r\bigl((x,y,0,\ldots,0),\mathbf{0}\bigr)$ only the first two rows of the board are nonempty. Write $X_j=V_{1j}$ and $Y_j=V_{2j}$, so that $|X_j|=x$ and $|Y_j|=y$. Two different squares $X_j$ and $X_l$ are joined by positive edges, and the same holds for $Y_j$ and $Y_l$. Each $X_j$ is joined to $Y_j$ by negative edges, and there are no other edges. Switching at all vertices of $Y_1\cup\cdots\cup Y_r$ makes every edge positive. The resulting graph is exactly $H_r(x,y)$ of Mojallal~\cite{Mojallal}.
\end{example}

We now list some basic properties of signed grid graphs. Here $\mathbf{1}_k$ denotes the
all-ones column vector of length $k$.

\begin{lemma}\label{lem-grid}
	Let $W$ be an $s\times t$ matrix of nonnegative integers and let $V$ be the vertex set of $\Gamma[W]$.
	\begin{enumerate}[label=(\roman*)]
		\item $A(\Gamma[W])=RR^\top-CC^\top$. Moreover, $R\mathbf{1}_s=C\mathbf{1}_t=\mathbf{1}_{|V|}$, and all diagonal entries of $RR^\top$ and $CC^\top$ are equal to $1$.
        
		\item A set $U\subseteq V$ is independent if and only if the squares meeting $U$ lie in pairwise distinct rows and pairwise distinct columns.
        
		\item $\Gamma[W^\top]=-\Gamma[W]$, where square $(i,j)$ of $W$ is identified with square $(j,i)$ of $W^\top$. In particular, the two signed graphs have the same underlying graph.
        
		\item Adding or deleting zero rows or zero columns of $W$ does not change $\Gamma[W]$ or $\mathcal{B}(W)$.
        
		\item If $\Gamma[W]$ is connected, then so is $\mathcal{B}(W)$.
	\end{enumerate}
\end{lemma}

\begin{proof}
	(i) Since every vertex lies in exactly one row, $(RR^\top)_{uv}=\sum_iR_{ui}R_{vi}$ is $1$ if $r(u)=r(v)$ and $0$ otherwise. Similarly, $(CC^\top)_{uv}$ is $1$ if $c(u)=c(v)$ and $0$ otherwise. Hence $(RR^\top-CC^\top)_{uv}$ equals $1$ if $u$ and $v$ share only a row, $-1$ if they share only a column, and $0$ if they share both or neither (in particular when $u=v$). This is exactly the adjacency rule of $\Gamma[W]$. Each row of $R$ and of $C$ has a single entry $1$, which gives the remaining assertions.\\

	\noindent(ii) Two distinct vertices are non-adjacent precisely when they lie in the same square or have neither a row nor a column in common. So $U$ is independent if and only if no two of the squares meeting $U$ share a row or a column.\\

	\noindent(iii) Transposing $W$ interchanges the row and the column of each vertex. Hence the roles of $R$ and $C$ are exchanged, and by (i) the adjacency matrix becomes $CC^\top-RR^\top=-A(\Gamma[W])$.\\
	
	\noindent(iv) A zero row or column contains no vertices and is not a vertex of $\mathcal{B}(W)$.\\
	
	\noindent(v) Each nonempty square $(i,j)$ is an edge of $\mathcal{B}(W)$, and $\mathcal{B}(W)$ has no isolated vertices. Suppose $\mathcal{B}(W)$ is disconnected, let $K$ be one of its components, and let $U$ be the set of vertices of $\Gamma[W]$ whose square is an edge of $K$. Then $U\neq\emptyset$, and $V\setminus U\neq\emptyset$ since another component also contains an edge. If $u\in U$ and $v\notin U$, their squares lie in different components of $\mathcal{B}(W)$, so they share no row and no column, and $u,v$ are non-adjacent. Therefore $\Gamma[W]$ is disconnected.
\end{proof}

Next we compute the independence number of $\Gamma[W]$ by means of rook placements. We write $[k]=\{1,\ldots,k\}$.

\begin{lemma}\label{lem-rook}
	Let $W$ be an $s\times t$ matrix of nonnegative integers with $s\leq t$, and let $n=\sum_{i,j}w_{ij}$. For an injection function $\varphi:[s]\to[t]$ put $\eta(\varphi)=\sum_{i=1}^{s}w_{i\varphi(i)}$; this is the number of vertices on the squares occupied by the $s$ non-attacking rooks placed at $(i,\varphi(i))$. Then
	\begin{enumerate}[label=(\alph*)]
		\item $\alpha(\Gamma[W])=\max_{\varphi}\eta(\varphi)$;
        
		\item the average of $\eta(\varphi)$ over all injections $\varphi$ is $n/t$; hence
		$\alpha(\Gamma[W])\geq n/t$;
        
		\item $\alpha(\Gamma[W])=n/t$ if and only if $\eta$ is constant;
        
		\item $\eta$ is constant if and only if either $s<t$ and $w_{ij}=a_i$ for some integers $a_i$, or $s=t$ and $w_{ij}=a_i+b_j$ for some integers $a_i,b_j$.
	\end{enumerate}
\end{lemma}

\begin{proof}
	(a) For each injection $\varphi$, the set $U_\varphi=\bigcup_iV_{i\varphi(i)}$ is independent by Lemma~\ref{lem-grid}(ii), and $|U_\varphi|=\eta(\varphi)$. Conversely, let $U$ be independent. By Lemma~\ref{lem-grid}(ii), the squares meeting $U$ are $(i_1,j_1),\ldots,(i_k,j_k)$ with distinct $i$'s and distinct $j$'s. Since $t-k\geq s-k$, the remaining $s-k$ rows can be mapped injectively into the $t-k$ unused columns. This extends $i_r\mapsto j_r$ to an injection $\varphi$ with $U\subseteq U_\varphi$, so $|U|\leq\eta(\varphi)$.\\

	\noindent(b) There are $t!/(t-s)!$ injections, and exactly $(t-1)!/(t-s)!$ of them satisfy $\varphi(i)=j$ for a fixed pair $(i,j)$, that is, a fraction $1/t$ of all. Hence the average of $\eta$ is $\frac1t\sum_{i,j}w_{ij}=n/t$, and the maximum is at least this value.\\
	
	\noindent(c) The maximum of finitely many numbers equals their average if and only if they are all equal.\\
	
	\noindent(d) If $s<t$ and $w_{ij}=a_i$, then $\eta\equiv\sum_ia_i$. If $s=t$ and $w_{ij}=a_i+b_j$, then every injection is a permutation and $\eta\equiv\sum_ia_i+\sum_jb_j$.\\[-3mm]
	
	Conversely, suppose $\eta$ is constant. Let $s<t$. Fix a row $i$ and columns $j\neq j'$. As $s-1\leq t-2$, the other $s-1$ rows can be mapped injectively into $[t]\setminus\{j,j'\}$. Extending this map by $i\mapsto j$ or by $i\mapsto j'$ gives two injections of equal weight, so $w_{ij}=w_{ij'}$. Hence $w_{ij}=a_i$ for all $i,j$.\\[-3mm]
	
	Let $s=t$. If $s=1$, take $a_1=w_{11}$ and $b_1=0$. If $s\geq2$, let $i\neq k$ and $j\neq l$. Choose a permutation $\pi$ with $\pi(i)=j$ and $\pi(k)=l$, and let $\pi'$ agree with $\pi$ except that $\pi'(i)=l$ and $\pi'(k)=j$. Then $\eta(\pi)=\eta(\pi')$ gives
	\begin{equation}\tag{E}\label{eq-E}
		w_{ij}+w_{kl}=w_{il}+w_{kj}.
	\end{equation}
	This identity is trivial when $i=k$ or $j=l$, so it holds for all $i,j,k,l$. Setting
	$a_i=w_{i1}$ and $b_j=w_{1j}-w_{11}$, identity \eqref{eq-E} with $k=l=1$ gives
	$w_{ij}=a_i+b_j$.
\end{proof}

\begin{remark}\label{rem-nonneg}
	The integers $a_i,b_j$ in Lemma~\ref{lem-rook}(d) may be chosen nonnegative. Indeed, let
	$b_{j_0}=\min_jb_j$ and replace $a_i$ by $a_i+b_{j_0}$ and $b_j$ by $b_j-b_{j_0}$. This does not
	change $w_{ij}$; the new $b_j$ are nonnegative, and the new $a_i$ equal $w_{ij_0}\geq0$. Hence,
	if $W$ is square with $w_{ij}=a_i+b_j$, then $\Gamma[W]$ is an additive rook graph.
\end{remark}

\begin{corollary}\label{cor-rookinv}
	The additive rook graph $\Rook_N(a,b)$ has $n=N\ell$ vertices, $\alpha=\ell$ and
	$\tau=(N-1)\ell$.
\end{corollary}

\begin{proof}
	We have $n=\sum_{i,j}(a_i+b_j)=N\ell$. Every permutation $\pi$ of $[N]$ has weight
	$\sum_ia_i+\sum_ib_{\pi(i)}=\ell$, so $\alpha=\ell$ by Lemma~\ref{lem-rook}(a), and
	$\tau=n-\alpha=(N-1)\ell$.
\end{proof}

We now recall the main result of this article, stated in the introduction.

\begin{mainthmrestate}
Let $\Gamma=(G,\sigma)$ be a signed graph. Then $\mathcal{E}(\Gamma)\geq2\tau(G)$. Equality
holds if and only if every connected component of $\Gamma$ is switching equivalent to an
additive rook graph $\Rook_N(a,b)$.
\end{mainthmrestate}

The inequality is proved in Section~\ref{sec-bound}, and the characterization of equality in Section~\ref{sec-equality}.

\section{Proof of $\mathcal{E}(\Gamma)\geq2\tau(G)$}\label{sec-bound}

In this section we prove the inequality part of Theorem~\ref{thm-main}. We first collect some results from linear algebra and matrix theory that are needed in the proof. Few parts of the argument follow/adapts the approach of Kumar and Pragada~\cite{KumarPragada} for unsigned graphs. For the sake of completeness, we include all the details. Our main novel contribution, the characterization of equality, is established in the later sections~\ref{sec-equality}.

\begin{lemma}\label{lem-var}
Let $M$ be a real symmetric $n\times n$ matrix. If $M=X-Y$ with $X,~Y ~\succcurlyeq0$, then
\[
\mathcal{E}(M)\leq\operatorname{tr}X+\operatorname{tr}Y .
\]
\end{lemma}

\begin{proof}
By (F1), let $u_1,\ldots,u_n$ be an orthonormal basis of eigenvectors of $M$ with
$Mu_k=\lambda_ku_k$. Let $U$ be the $n\times n$ matrix with columns $u_1,\ldots,u_n$. The $(j,k)$-th entry of $U^{\top}U$ is $u_j^{\top}u_k$, which equals $1$ if $j=k$ and $0$ otherwise. Hence $U^{\top}U=I$ and therefore, for any $n\times n$ matrix $Z$, 
\[
\tr (Z)=\tr(ZI)=\tr(ZU^TU)=\operatorname{tr}(U^{\top}ZU)=\sum_{k=1}^{n}u_k^{\top}Zu_k .
\]
Therefore,
$\operatorname{tr}X=\sum_{k=1}^{n}u_k^{\top}Xu_k$ and $\operatorname{tr}Y=\sum_{k=1}^{n}u_k^{\top}Yu_k$. As $X\succcurlyeq0$ and $Y\succcurlyeq0$, every term $u_k^{\top}Xu_k$ and $u_k^{\top}Yu_k$ is nonnegative. Hence
\[
\operatorname{tr}X\;\geq\;\sum_{\lambda_k>0}u_k^{\top}Xu_k
\;\geq\;\sum_{\lambda_k>0}u_k^{\top}(X-Y)u_k=\sum_{\lambda_k>0}u_k^{\top}Mu_k
=\sum_{\lambda_k>0}\lambda_k ,
\]
and similarly
\[
\operatorname{tr}Y\;\geq\;\sum_{\lambda_k<0}u_k^{\top}Yu_k
\;\geq\;\sum_{\lambda_k<0}u_k^{\top}(Y-X)u_k=-\sum_{\lambda_k<0}\lambda_k
=\sum_{\lambda_k<0}|\lambda_k| .
\]
Adding the two inequalities gives $\operatorname{tr}X+\operatorname{tr}Y\geq\sum_{k}|\lambda_k|=\mathcal{E}(M)$.
\end{proof}


To simplify notation, we fix the following for the rest of the article. For a signed graph $\Gamma=(G,\sigma)$, let $A=A(\Gamma)$, and let $P=A_{+}$, $Q=A_{-}$ and $B=|A|=P+Q$ be as in (F2). By (F3), we also fix Gram representations $P_{uv}=\langle p_u,p_v\rangle$ and $Q_{uv}=\langle q_u,q_v\rangle$ for all vertices $u,v$.

\begin{lemma}\label{lem-parts}
Let $\Gamma$ be a signed graph, and let $A$, $P$, $Q$, $B$ and the Gram vectors $p_v,q_v$ be as fixed above. Then the following hold.
\begin{enumerate}[label=(\arabic*)]
  \item $A=P-Q$, $B=P+Q$, $PQ=QP=0$, $B^2=A^2$ and $\operatorname{tr}B=\mathcal{E}(\Gamma)$.
  
  \item $P_{vv}=Q_{vv}=\tfrac12B_{vv}$ for every vertex $v$. If $u\neq v$ are
  non-adjacent, then $P_{uv}=Q_{uv}=\tfrac12B_{uv}$.
  
  \item $(B^2)_{vv}=d_v$ for every vertex $v$.
  
  \item Let $D=\operatorname{diag}(\delta_v)$ with $\delta_v\in\{1,-1\}$, and let $\Gamma^D$ be the signed graph with $A(\Gamma^D)=DAD$. Then $A(\Gamma^D)_{\pm}=DA_{\pm}D$ and $|A(\Gamma^D)|=D|A|D$. Consequently, the vectors $\delta_vp_v$ and $\delta_vq_v$ form Gram representations of $A(\Gamma^D)_{+}$ and $A(\Gamma^D)_{-}$, respectively.
\end{enumerate}
\end{lemma}

\begin{proof}
(1) All assertions except the last follow from (F2), and $\operatorname{tr}B=\mathcal{E}(\Gamma)$ is the formula for the energy in (F2). Since $PQ=QP=0$, we get $B^2=(P+Q)^2=P^2+Q^2=(P-Q)^2=A^2$.\\[-2mm]

\noindent(2) By (1), $P=\tfrac12(B+A)$ and $Q=\tfrac12(B-A)$. Since $A_{vv}=0$, this gives $P_{vv}=Q_{vv}=\tfrac12B_{vv}$. If $u\neq v$ are non-adjacent, then $A_{uv}=0$, and in the same way $P_{uv}=Q_{uv}=\tfrac12B_{uv}$.\\[-2mm]

\noindent(3) By (1), $(B^2)_{vv}=(A^2)_{vv}=\sum_{u}A_{uv}^2$. Each neighbour $u$ of $v$ contributes $\sigma(uv)^2=1$, and every other vertex contributes $0$. Hence $(B^2)_{vv}=d_v$.\\[-2mm]

\noindent(4) Write $A=\sum_k\lambda_ku_ku_k^{\top}$ as in (F1). Then $DAD=\sum_k\lambda_k(Du_k)(Du_k)^{\top}$, and the vectors $Du_k$ are again orthonormal because $D^{\top}D=I$. So this is a spectral decomposition of $DAD$ with the same eigenvalues, and therefore
\[
A(\Gamma^D)_{+}=\sum_{\lambda_k>0}\lambda_kDu_ku_k^{\top}D=DPD,\qquad
A(\Gamma^D)_{-}=DQD,\qquad |A(\Gamma^D)|=DBD .
\]
Finally, $(DPD)_{uv}=\delta_u\delta_vP_{uv}=\langle\delta_up_u,\delta_vp_v\rangle$, and similarly $(DQD)_{uv}=\langle\delta_uq_u,\delta_vq_v\rangle$.
\end{proof}


The number $B_{vv}=|A|_{vv}$ is the \emph{vertex energy} $\En_\Gamma(v)$ of~\cite{ArizmendiVertex}; note that $\sum_v\En_\Gamma(v)=\En(\Gamma)$.

\begin{lemma}\label{lem-edge}
If $uv\in E$, then $B_{uu}B_{vv}\ge(1+|B_{uv}|)^2$. In particular $B_{vv}>0$ for every non-isolated vertex $v$.
\end{lemma}

\begin{proof}
Let $x=B_{uu}$, $y=B_{vv}$, $z=B_{uv}$ and $\sgn=A_{uv}\in\{\pm1\}$. The matrices $2P=B+A$ and $2Q=B-A$ are PSD, and their principal submatrices on $\{u,v\}$ are
\[
\begin{pmatrix}x&z+\sgn\\ z+\sgn&y\end{pmatrix},\qquad
\begin{pmatrix}x&z-\sgn\\ z-\sgn&y\end{pmatrix}.
\]
By (P2), $xy\ge(z\pm\sgn)^2$. One of $z+\sgn$ and $z-\sgn$ has absolute value $1+|z|$. Hence $xy\ge(1+|z|)^2\ge1$, and since $x,y\ge0$ we get $x,y>0$.
\end{proof}


\begin{lemma}\label{lem-local}
Let $v$ be non-isolated and $S(v):=V\setminus N[v]$. Then
\[
\En(\Gamma-N[v])\le\sum_{u\in S(v)}B_{uu}-\frac1{B_{vv}}\sum_{u\in S(v)}B_{uv}^2 .
\]
\end{lemma}

\begin{proof}
Put $\beta=P_{vv}=Q_{vv}=\tfrac12B_{vv}>0$ (Lemmas~\ref{lem-parts} and~\ref{lem-edge}). For $u\in S(v)$ define the projections
\[
p'_u=p_u-\frac{P_{uv}}{\beta}\,p_v,\qquad q'_u=q_u-\frac{Q_{uv}}{\beta}\,q_v .
\]
Since $\|p_v\|^2=\beta$, for $u,w\in S(v)$ we have $\langle p'_u,p'_w\rangle=P_{uw}-P_{uv}P_{wv}/\beta$, and similarly for $q'$. As $P_{uv}=Q_{uv}$ and $P_{wv}=Q_{wv}$ by Lemma~\ref{lem-parts}(2),
\[
\langle p'_u,p'_w\rangle-\langle q'_u,q'_w\rangle=P_{uw}-Q_{uw}=A_{uw}.
\]
So the Gram matrices $X$ of $(p'_u)$ and $Y$ of $(q'_u)$ are PSD with $X-Y=A(\Gamma-N[v])$. By Lemma~\ref{lem-var},
\[
\En(\Gamma-N[v])\le\sum_{u\in S(v)}\Big(P_{uu}+Q_{uu}-\frac{P_{uv}^2+Q_{uv}^2}{\beta}\Big)
=\sum_{u\in S(v)}\Big(B_{uu}-\frac{B_{uv}^2}{B_{vv}}\Big).
\]
If $S(v)=\emptyset$, both sides are $0$.
\end{proof}


\begin{lemma}\label{lem-edge}
If $uv\in E$, then $B_{uu}B_{vv}\geq(1+|B_{uv}|)^2$. In particular, $B_{vv}>0$ for every
non-isolated vertex $v$.
\end{lemma}

\begin{proof}
Let $x=B_{uu}$, $y=B_{vv}$, $z=B_{uv}$ and $s=A_{uv}\in\{1,-1\}$. By Lemma~\ref{lem-parts}(1), the matrices $2P=B+A$ and $2Q=B-A$ are positive semidefinite. Since $A_{uu}=A_{vv}=0$, their principal submatrices on $\{u,v\}$ are
\[
\begin{pmatrix}x&z+s\\ z+s&y\end{pmatrix}
\qquad\text{and}\qquad
\begin{pmatrix}x&z-s\\ z-s&y\end{pmatrix}.
\]
Both have nonnegative determinant, so $xy\geq(z+s)^2$ and $xy\geq(z-s)^2$. As $s=\pm1$, one of the numbers $z+s$ and $z-s$ has absolute value $1+|z|$. Hence $xy\geq(1+|z|)^2\geq1$. Since $x$ and $y$ are diagonal entries of $B\succcurlyeq0$, they are nonnegative, and therefore $x,y>0$.
\end{proof}


\begin{lemma}\label{lem-local}
Let $v$ be a non-isolated vertex and put $S(v)=V\setminus N[v]$. Then
\[
\mathcal{E}(\Gamma-N[v])\leq\sum_{u\in S(v)}B_{uu}-\frac{1}{B_{vv}}\sum_{u\in S(v)}B_{uv}^2 .
\]
\end{lemma}
\begin{proof}
If $S(v)=\emptyset$, both sides are $0$. Otherwise, let $\beta=P_{vv}=Q_{vv}=\tfrac12B_{vv}$, which is positive by Lemmas~\ref{lem-parts}(2) and~\ref{lem-edge}. For $u\in S(v)$, let $p'_u$ and $q'_u$ be the components of $p_u$ and $q_u$ orthogonal to $p_v$ and $q_v$, respectively:
\[
p'_u=p_u-\frac{P_{uv}}{\beta}\,p_v,\qquad q'_u=q_u-\frac{Q_{uv}}{\beta}\,q_v .
\]
Since $\|p_v\|^2=\|q_v\|^2=\beta$, for all $u,w\in S(v)$ we get
\[
\langle p'_u,p'_w\rangle=P_{uw}-\frac{P_{uv}P_{wv}}{\beta},\qquad
\langle q'_u,q'_w\rangle=Q_{uw}-\frac{Q_{uv}Q_{wv}}{\beta}.
\]
The vertices $u,w\in S(v)$ are not adjacent to $v$, so $P_{uv}=Q_{uv}$ and $P_{wv}=Q_{wv}$ by Lemma~\ref{lem-parts}(2). Hence $\langle p'_u,p'_w\rangle-\langle q'_u,q'_w\rangle=P_{uw}-Q_{uw}=A_{uw}$. Thus the Gram matrices $X$ of $(p'_u)$ and $Y$ of $(q'_u)$ are positive semidefinite and satisfy $X-Y=A(\Gamma-N[v])$. By Lemma~\ref{lem-var},
\[
\mathcal{E}(\Gamma-N[v])\leq\operatorname{tr}X+\operatorname{tr}Y
=\sum_{u\in S(v)}\Bigl(P_{uu}+Q_{uu}-\frac{P_{uv}^2+Q_{uv}^2}{\beta}\Bigr)
=\sum_{u\in S(v)}\Bigl(B_{uu}-\frac{B_{uv}^2}{B_{vv}}\Bigr),
\]
where the last step uses $P_{uv}=Q_{uv}=\tfrac12B_{uv}$ and $P_{uu}+Q_{uu}=B_{uu}$.
\end{proof}


\begin{lemma}\label{lem-sum}
Let $\Gamma$ be a signed graph with no isolated vertices. For vertices $u,v$ write $x_v=B_{vv}$ and $z_{uv}=B_{uv}$. By Lemma~\ref{lem-edge}, $x_v>0$ for every $v$. For an edge $uv$ put
\[
F_{uv}=\frac{(x_u-1)^2-z_{uv}^2}{x_u}+\frac{(x_v-1)^2-z_{uv}^2}{x_v}.
\]
Then
\[
\sum_{v}\mathcal{E}(\Gamma-N[v])\leq n\,\mathcal{E}(\Gamma)-4m-\sum_{uv\in E}F_{uv}.
\]
\end{lemma}

\begin{proof}
Summing Lemma~\ref{lem-local} over all vertices $v$ gives $\sum_v\mathcal{E}(\Gamma-N[v])\leq T_1-T_2$, where
\[
T_1=\sum_{v}\sum_{u\in S(v)}x_u
\qquad\text{and}\qquad
T_2=\sum_{v}\frac{1}{x_v}\sum_{u\in S(v)}z_{uv}^2 .
\]

For a fixed vertex $u$, we have $u\in S(v)$ for exactly $n-1-d_u$ vertices $v$. Since
$\sum_ux_u=\operatorname{tr}B=\mathcal{E}(\Gamma)$, this gives
\[
T_1=\sum_{u}(n-1-d_u)\,x_u=n\,\mathcal{E}(\Gamma)-\sum_{u}(d_u+1)\,x_u .
\]

By Lemma~\ref{lem-parts}(3), splitting the sum according to whether $u=v$, $u\in N(v)$ or
$u\in S(v)$, we get
\[
d_v=\sum_{u}B_{uv}^2=x_v^2+\sum_{u\sim v}z_{uv}^2+\sum_{u\in S(v)}z_{uv}^2 .
\]
Therefore
\[
T_2=\sum_{v}\Bigl(\frac{d_v}{x_v}-x_v-\sum_{u\sim v}\frac{z_{uv}^2}{x_v}\Bigr).
\]

The terms $\pm x_v$ cancel in $T_1-T_2$, and
\[
T_1-T_2=n\,\mathcal{E}(\Gamma)-\sum_{v}\Bigl(d_v\Bigl(x_v+\frac{1}{x_v}\Bigr)
-\sum_{u\sim v}\frac{z_{uv}^2}{x_v}\Bigr).
\]
Each edge $uv$ is counted once at $u$ and once at $v$. Hence, with
$h(x,z)=x+\frac{1}{x}-\frac{z^2}{x}$,
\[
T_1-T_2=n\,\mathcal{E}(\Gamma)-\sum_{uv\in E}\bigl(h(x_u,z_{uv})+h(x_v,z_{uv})\bigr).
\]
Finally, $h(x,z)-2=\frac{(x-1)^2-z^2}{x}$, so the last sum equals $4m+\sum_{uv\in E}F_{uv}$.
\end{proof}


\begin{lemma}\label{lem-F}
For every edge $uv$ we have $F_{uv}\geq0$, with equality if and only if $x_u=x_v=1+|z_{uv}|$.
\end{lemma}

\begin{proof}
Write $x=x_u$, $y=x_v$ and $z=z_{uv}$. Since $\frac{(x-1)^2-z^2}{x}=x-2+\frac{1-z^2}{x}$, a direct computation gives
\[
F_{uv}=(x+y)+(1-z^2)\Bigl(\frac1x+\frac1y\Bigr)-4=\frac{(x+y)(xy+1-z^2)}{xy}-4 .
\]
By Lemma~\ref{lem-edge}, $xy\geq(1+|z|)^2$, that is, $|z|\leq\sqrt{xy}-1$. Hence
\[
xy+1-z^2\geq xy+1-\bigl(\sqrt{xy}-1\bigr)^2=2\sqrt{xy}>0 .
\]
Together with $x+y\geq2\sqrt{xy}$, this gives $(x+y)(xy+1-z^2)\geq4xy$, and so $F_{uv}\geq0$. Since both factors are positive, equality holds if and only if $x+y=2\sqrt{xy}$ and $|z|=\sqrt{xy}-1$. This means $x=y$ and $|z|=x-1$, that is, $x=y=1+|z|$.
\end{proof}

We can now prove the inequality in Theorem~\ref{thm-main}.

\begin{theorem}\label{thm-ineq}
Let $\Gamma$ be a signed graph with vertex cover number $\tau(\Gamma)$. Then $\mathcal{E}(\Gamma)\geq2\tau(\Gamma)$.
\end{theorem}

\begin{proof}
We use induction on the number $n$ of vertices. For $n=0$ both sides are $0$. Let $n\geq1$, and write $\alpha=\alpha(\Gamma)$ and $\tau=\tau(\Gamma)$.\\[-2mm]

\noindent\textbf{Case 1:} $\Gamma$ has an isolated vertex $v$.\\[1mm]
Then $A(\Gamma)=A(\Gamma-v)\oplus(0)$, so
$\mathcal{E}(\Gamma)=\mathcal{E}(\Gamma-v)$. Moreover $\alpha=\alpha(\Gamma-v)+1$, and hence $\tau=\tau(\Gamma-v)$. The claim follows from the induction hypothesis applied to $\Gamma-v$.\\[-2mm]

\noindent\textbf{Case 2:} $\Gamma$ has no isolated vertex.\\[1mm] Let $v$ be a vertex. Adding $v$ to an independent set of $\Gamma-N[v]$ gives an independent set of $\Gamma$, so $\alpha(\Gamma-N[v])\leq\alpha-1$. Since $\Gamma-N[v]$ has $n-1-d_v$ vertices,
\[
\tau(\Gamma-N[v])=(n-1-d_v)-\alpha(\Gamma-N[v])\geq\tau-d_v .
\]
By the induction hypothesis, $\mathcal{E}(\Gamma-N[v])\geq2\tau-2d_v$. Summing over all vertices $v$ and using $\sum_vd_v=2m$, we obtain
\begin{equation}\label{eq-sumlower}
\sum_{v}\mathcal{E}(\Gamma-N[v])\geq2n\tau-4m .
\end{equation}
On the other hand, Lemmas~\ref{lem-sum} and~\ref{lem-F} give $\sum_v\mathcal{E}(\Gamma-N[v])\leq n\,\mathcal{E}(\Gamma)-4m$. Combining the two inequalities yields $n\,\mathcal{E}(\Gamma)\geq2n\tau$, and the result follows since $n\geq1$.
\end{proof}

\begin{corollary}\label{cor-comp}
Let $\Gamma_1,\ldots,\Gamma_k$ be the connected components of $\Gamma$. Then $\mathcal{E}(\Gamma)=2\tau(\Gamma)$ if and only if $\mathcal{E}(\Gamma_i)=2\tau(\Gamma_i)$ for every $i$.
\end{corollary}

\begin{proof}
Since $A(\Gamma)$ is block diagonal, $\mathcal{E}(\Gamma)=\sum_i\mathcal{E}(\Gamma_i)$. An
independent set of $\Gamma$ is a union of independent sets of the components, so
$\alpha(\Gamma)=\sum_i\alpha(\Gamma_i)$ and $\tau(\Gamma)=\sum_i\tau(\Gamma_i)$. Hence
$\mathcal{E}(\Gamma)-2\tau(\Gamma)=\sum_i\bigl(\mathcal{E}(\Gamma_i)-2\tau(\Gamma_i)\bigr)$. By Theorem~\ref{thm-ineq} every term of this sum is nonnegative. So the sum is zero if and only if every term is zero.
\end{proof}

\section{Characterization of the equality case}\label{sec-equality}
In this section we first show that every additive rook graph attains equality in \eqref{eq-bound-signed}. We then prove the converse: if a signed graph attains equality, then each of its connected components is switching equivalent to an additive rook graph.

\subsection{Additive rook graphs attain equality}

\begin{theorem}\label{thm-suff}
If $\Gamma$ is switching equivalent to $\Rook_N(a,b)$, then $\mathcal{E}(\Gamma)=2\tau(\Gamma)$.
\end{theorem}

\begin{proof}
Switching changes neither the spectrum nor the underlying graph, so it changes neither
$\mathcal{E}(\Gamma)$ nor $\tau(\Gamma)$. Hence we may assume that $\Gamma=\Rook_N(a,b)$. By Corollary~\ref{cor-rookinv}, $n=N\ell$ and $\alpha(\Gamma)=\ell=n/N$. Put
\[
P'=RR^{\top}-\tfrac1NJ\qquad\text{and}\qquad Q'=CC^{\top}-\tfrac1NJ .
\]
For $x\in\mathbb{R}^{n}$ let $y=R^{\top}x\in\mathbb{R}^{N}$. Since $R\mathbf{1}_N=\mathbf{1}_{n}$, we have $\mathbf{1}_{n}^{\top}x=\mathbf{1}_N^{\top}y$. The Cauchy--Schwarz inequality $(\mathbf{1}_N^{\top}y)^2\leq N\|y\|^2$ then gives
\[
x^{\top}P'x=\|y\|^2-\tfrac1N\bigl(\mathbf{1}_N^{\top}y\bigr)^2\geq0 .
\]
Thus $P'\succcurlyeq0$, and the same argument with $C$ in place of $R$ shows $Q'\succcurlyeq0$. By Lemma~\ref{lem-grid}(i), $P'-Q'=RR^{\top}-CC^{\top}=A(\Gamma)$, and since the diagonal entries of $RR^{\top}$ and $CC^{\top}$ are equal to $1$,
\[
\operatorname{tr}P'=\operatorname{tr}Q'=n-\frac nN=n-\alpha(\Gamma)=\tau(\Gamma).
\]
Lemma~\ref{lem-var} now gives $\mathcal{E}(\Gamma)\leq\operatorname{tr}P'+\operatorname{tr}Q'=2\tau(\Gamma)$, and Theorem~\ref{thm-ineq} gives the reverse inequality.
\end{proof}

\subsection{Equality forces additive rook graphs}

We now prove the converse of Theorem~\ref{thm-suff}, which is the main part of the article based on several novel ideas.

\begin{theorem}\label{thm-nec}
Let $\Gamma$ be a connected signed graph with $n\geq2$ vertices. If
$\mathcal{E}(\Gamma)=2\tau(\Gamma)$, then $\Gamma$ is switching equivalent to an additive rook graph $\Rook_N(a,b)$ for some $N\geq2$ and some $a,b\in\mathbb{Z}_{\geq0}^{N}$.
\end{theorem}

Throughout this subsection, $\Gamma$ is a connected signed graph with $n\geq2$ vertices and $\mathcal{E}(\Gamma)=2\tau(\Gamma)$, and $A$, $P$, $Q$, $B$, $p_v$, $q_v$ are as in
Lemma~\ref{lem-parts}. We write $\alpha=\alpha(\Gamma)$ and $\tau=\tau(\Gamma)$. Since $\Gamma$ has an edge, $\alpha\geq1$ and $\tau\geq1$. Put
\[
d:=\frac{\tau}{n}\qquad\text{and}\qquad e:=\frac{\alpha}{n}.
\]
Then $d+e=1$ and $d,e\in(0,1)$. Switching does not change connectivity, $n$, $\tau$ or the
energy. Hence everything below also holds for every signed graph switching equivalent to $\Gamma$.\\[-1mm]

For an extremal signed graph $\Gamma$ with $n$ vertices, the independent number $\alpha \leq \frac{n}{2}$ is obtained in the next lemma.

\begin{proposition}\label{prop-rigid}
The following hold.
\begin{enumerate}[label=(\alph*)]
  \item $B_{vv}=2d$ for every vertex $v$.
  \item $|B_{uv}|=d-e$ for every edge $uv$.
  \item $d\geq e$, that is, $\alpha\leq n/2$.
\end{enumerate}
\end{proposition}

\begin{proof}
Since $\Gamma$ is connected and $n\geq2$, it has no isolated vertices. By \eqref{eq-sumlower}, Lemma~\ref{lem-sum} and $\mathcal{E}(\Gamma)=2\tau$,
\[
2n\tau-4m\leq\sum_{v}\mathcal{E}(\Gamma-N[v])\leq n\,\mathcal{E}(\Gamma)-4m-\sum_{uv\in E}F_{uv}=2n\tau-4m-\sum_{uv\in E}F_{uv}.
\]
Thus $\sum_{uv\in E}F_{uv}\leq0$. By Lemma~\ref{lem-F}, every $F_{uv}$ is zero, and therefore $x_u=x_v=1+|z_{uv}|$ for every edge $uv$. Since $\Gamma$ is connected, all $x_v$ are equal to a common value $x$. Then $nx=\operatorname{tr}B=\mathcal{E}(\Gamma)=2\tau$, so $x=2d$. This proves (a). For every edge, $|z_{uv}|=x-1=2d-(d+e)=d-e$, which proves (b). Finally, (c) follows from $|z_{uv}|\geq0$.
\end{proof}


\begin{lemma}\label{lem-values} The following hold.
\begin{enumerate}[label=(\arabic*)]
  \item $\|p_v\|^2=\|q_v\|^2=d$ for every vertex $v$.
  \item For every edge $uv$, the pair $(P_{uv},Q_{uv})$ equals either
  $(A_{uv}d,\,-A_{uv}e)$ or $(A_{uv}e,\,-A_{uv}d)$.
  \item For distinct non-adjacent vertices $u,v$, $P_{uv}=Q_{uv}$.
  \item $|P_{uv}|=d$ if and only if $p_u=\pm p_v$, and $|Q_{uv}|=d$ if and only if
  $q_u=\pm q_v$.
\end{enumerate}
\end{lemma}

\begin{proof}
(1) follows from Lemma~\ref{lem-parts}(2) and Proposition~\ref{prop-rigid}(a).

\noindent(2) Let $s=A_{uv}\in\{1,-1\}$. By Proposition~\ref{prop-rigid}(b), $B_{uv}=\pm s(d-e)$. We use
$P=\tfrac12(B+A)$, $Q=\tfrac12(B-A)$ and $1=d+e$. If $B_{uv}=s(d-e)$, then $P_{uv}=\tfrac s2(d-e+d+e)=sd$ and $Q_{uv}=\tfrac s2(d-e-d-e)=-se$. If $B_{uv}=-s(d-e)$, the same
computation gives $P_{uv}=se$ and $Q_{uv}=-sd$.\\[-1mm]

\noindent(3) This is Lemma~\ref{lem-parts}(2).\\[-1mm]

\noindent(4) This follows from (1) and (F4) with $r=\sqrt d$.
\end{proof}


\begin{definition}\label{def-classes}
Let $\Gamma$ be a signed graph with vertex set $V$, and let $P=A(\Gamma)_{+}$ and $Q=A(\Gamma)_{-}$ have the Gram representations $P_{uv}=\langle p_u,p_v\rangle$ and $Q_{uv}=\langle q_u,q_v\rangle$. We define two relations on $V$ by
\[
u\approx_Pv\iff p_u=\pm p_v
\qquad\text{and}\qquad
u\approx_Qv\iff q_u=\pm q_v .
\]
Both are equivalence relations. Their equivalence classes are called \emph{$\approx_P$-classes} and \emph{$\approx_Q$-classes}, respectively. Each of the two families of classes partitions $V$.
\end{definition}


By Proposition~\ref{prop-rigid}(c), every extremal signed graph satisfies $\alpha\leq n/2$. We first treat the boundary case $\alpha=n/2$.

\begin{proposition}\label{prop-boundary}
If $d=e$, that is, $\alpha=n/2$, then $n=2a$ for some $a\geq1$, and $\Gamma$ is switching equivalent to $K_{a,a}=\Rook_2\bigl((a,0),(0,0)\bigr)$.
\end{proposition}

\begin{proof}
Here $d=e=\tfrac12$. By Lemma~\ref{lem-values}(2), $|P_{uv}|=\tfrac12=d$ on every edge $uv$, so $u\approx_Pv$ by Lemma~\ref{lem-values}(4). Since $\Gamma$ is connected, $\approx_P$ has only one class. Fix a vertex $w$ and write $p_v=\delta_vp_w$ with $\delta_v\in\{1,-1\}$. Let $D=\operatorname{diag}(\delta_v)$ and $\Gamma'=\Gamma^D$. By Lemma~\ref{lem-parts}(4), the vectors $\delta_vp_v=p_w$ form a Gram representation of $P'=A(\Gamma')_{+}$, and hence $P'=\tfrac12J$. Let $q'_v$ be Gram vectors of $Q'=A(\Gamma')_{-}=\tfrac12J-A(\Gamma')$.\\[-3mm]

By Lemma~\ref{lem-values}(1) applied to $\Gamma'$, we have $\|q'_v\|^2=\tfrac12$. So the Cauchy--Schwarz inequality gives $|Q'_{uv}|\leq\tfrac12$, and therefore $A(\Gamma')_{uv}=\tfrac12-Q'_{uv}\in[0,1]$ for $u\neq v$. Thus all edges of $\Gamma'$ are positive. By (F4), $Q'_{uv}=\tfrac12$ gives $q'_u=q'_v$ on non-edges, and $Q'_{uv}=-\tfrac12$ gives $q'_u=-q'_v$ on edges.\\[-3mm]

Let $X=\{v:q'_v=q'_w\}$ and $Y=\{v:q'_v=-q'_w\}$. Every vertex lies in exactly one of these sets, and $Y\neq\emptyset$ because $w$ has a neighbour. Two vertices are adjacent if and only if they lie in different sets. Hence $\Gamma'$ is the all-positive complete bipartite graph $K_{|X|,|Y|}$. Since $\alpha=\max(|X|,|Y|)=n/2$, we get $|X|=|Y|=a$ with $n=2a$. Finally, $\Rook_2\bigl((a,0),(0,0)\bigr)$ has two squares of size $a$ in row $1$ and no other vertices. They are joined by positive edges, so this graph is exactly $K_{a,a}$.
\end{proof}

We now turn to the remaining case $\alpha<n/2$ (i.e., $e<d$). The next lemma shows that the underlying graph of $\Gamma$ is then the underlying graph of a signed grid graph.

\begin{lemma}\label{lem-xor}
Let $\Gamma $ be a connected signed graph of $n$ vertices such that $\mathcal{E}(\Gamma)=2\tau(\Gamma).$   Suppose $e<d$ (i.e., $\alpha<n/2$). Then for distinct vertices $u,v$,
\[
u\sim v\iff\text{exactly one of } u\approx_Pv \text{ and } u\approx_Qv \text{ holds.}
\]
Consequently, let $R_1,\ldots,R_s$ be the $\approx_P$-classes and $C_1,\ldots,C_t$ the
$\approx_Q$-classes. Put $V_{ij}=R_i\cap C_j$ and let $W$ be the $s\times t$ matrix with
entries $w_{ij}=|V_{ij}|$. Regard $\Gamma[W]$ as a signed graph on $V$ whose square $(i,j)$ is $V_{ij}$. Then $\Gamma$ and $\Gamma[W]$ have the same edges, and they can differ only in the signs of these edges. Moreover, $W$ has no zero row and no zero column, and $\mathcal{B}(W)$ is connected.
\end{lemma}

\begin{proof}
Let $uv$ be an edge. By Lemma~\ref{lem-values}(2), $\{|P_{uv}|,|Q_{uv}|\}=\{d,e\}$, and
$d\neq e$. By Lemma~\ref{lem-values}(4), exactly one of $u\approx_Pv$ and $u\approx_Qv$ holds. Now let $u\neq v$ be non-adjacent. Then $P_{uv}=Q_{uv}$ by Lemma~\ref{lem-values}(3), so either both relations hold or neither does. This proves the equivalence.

By construction, $u\approx_Pv$ means that $u$ and $v$ lie in the same row, and $u\approx_Qv$ means that they lie in the same column. Comparing with Definition~\ref{def-grid}, we see that $\Gamma$ and $\Gamma[W]$ have the same edges. Each row and each column is an equivalence class, so it is nonempty. Since $\Gamma$ is connected, so is $\Gamma[W]$, and Lemma~\ref{lem-grid}(v) shows that $\mathcal{B}(W)$ is connected.
\end{proof}


Now we are ready to prove the following key result. By Lemma~\ref{lem-xor}, the underlying graph of an extremal signed graph with $e<d$ (i.e., $\alpha< \frac{n}{2}$) is the underlying graph of a signed grid graph $\Gamma[W]$. The next result shows that the signs are also those of $\Gamma[W]$, up to switching.

\begin{proposition}\label{prop-normal}
Let $\Gamma$ be a connected signed graph with $\mathcal{E}(\Gamma)=2\tau(\Gamma)$ and $e<d$, and let $W$ be as in Lemma~\ref{lem-xor}. Then $\Gamma\sim\Gamma[W]$. More precisely, there is a diagonal matrix $D$ with diagonal entries $\pm1$ such that
\[
A(\Gamma^D)=RR^{\top}-CC^{\top},\qquad P(\Gamma^D)=RR^{\top}-eJ,\qquad Q(\Gamma^D)=CC^{\top}-eJ ,
\]
where $R$ and $C$ are the row-incidence and column-incidence matrices of $\Gamma[W]$.
\end{proposition}

\begin{proof}
We use two facts throughout. First, by the proof of Lemma~\ref{lem-grid}(i), $(RR^{\top})_{uv}=1$ if $u,v$ lie in the same row and $(RR^{\top})_{uv}=0$ otherwise, and
similarly for $CC^{\top}$ and columns. Second, by Lemma~\ref{lem-xor}, two distinct vertices are adjacent if and only if they share exactly one of their row and their column. The proof is divided into eight claims.

\medskip\noindent\textbf{Claim 1.} \emph{After a switching, we may assume that there are vectors $\hat p_i$ and $\hat q_j$ and signs $\theta_v\in\{1,-1\}$ such that
\[
p_v=\hat p_{r(v)}\qquad\text{and}\qquad q_v=\theta_v\,\hat q_{c(v)}\qquad\text{for every } v .
\]
Moreover, $\|\hat p_i\|^2=\|\hat q_j\|^2=d$.}\\[2mm]
For each row $R_i$ pick $v_i\in R_i$ and put
$\hat p_i=p_{v_i}$. Since $R_i$ is a $\approx_P$-class, each $v\in R_i$ satisfies
$p_v=\delta_v\hat p_i$ for some $\delta_v\in\{1,-1\}$. Switch $\Gamma$ by $\operatorname{diag}(\delta_v)$. By Lemma~\ref{lem-parts}(4), the new Gram vectors are
$\delta_vp_v=\hat p_{r(v)}$ and $\delta_vq_v$. Changing the signs of Gram vectors does not change $\approx_P$ or $\approx_Q$, so the rows and columns are unchanged. We rename the switched graph $\Gamma$. For each column $C_j$ pick $w_j\in C_j$ and put $\hat q_j=q_{w_j}$. Since $C_j$ is a $\approx_Q$-class, $q_v=\theta_v\hat q_j$ with $\theta_v\in\{1,-1\}$ for each $v\in C_j$. The norms equal $d$ by Lemma~\ref{lem-values}(1). 

\medskip\noindent\textbf{Claim 2.} \emph{If $u,v$ lie in the same square, then
$\theta_u=\theta_v$.}\\[2mm]
We may assume $u\neq v$. Since $u,v$ share both their row and
their column, they are non-adjacent, so $P_{uv}=Q_{uv}$ by Lemma~\ref{lem-values}(3). By Claim~1, $P_{uv}=\|\hat p_i\|^2=d$ and $Q_{uv}=\theta_u\theta_v\|\hat q_j\|^2=\theta_u\theta_vd$. As $d>0$, we get $\theta_u\theta_v=1$. 

\medskip\noindent\textbf{Claim 3.} \emph{With $\Theta=\operatorname{diag}(\theta_v)$ we have $A=RR^{\top}-\Theta CC^{\top}\Theta$.}\\[2mm]
We compare the $(u,v)$ entries. If $u,v$ lie in the same square (including $u=v$), then $A_{uv}=0$, and the right-hand side is $1-\theta_u\theta_v=0$ by Claim~2. If $u,v$ share only a row, then $uv$ is an edge and the right-hand side is $1$. Here $P_{uv}=\|\hat p_{r(u)}\|^2=d$. By Lemma~\ref{lem-values}(2), $P_{uv}\in\{A_{uv}d,A_{uv}e\}$, and since $e\neq d$ this gives $A_{uv}=1$. If $u,v$ share only a column, then $uv$ is an edge and the right-hand side is $-\theta_u\theta_v$. Here $Q_{uv}=\theta_u\theta_vd$. By Lemma~\ref{lem-values}(2), $Q_{uv}\in\{-A_{uv}e,-A_{uv}d\}$. Since $|Q_{uv}|=d\neq e$, we get $Q_{uv}=-A_{uv}d$, so $A_{uv}=-\theta_u\theta_v$. In all other cases $u,v$ are non-adjacent and both sides are $0$. 

\medskip
We now consider the \emph{defect matrix} $X=RR^{\top}-P$. It is symmetric. Since $P=A+Q$,
Claim~3 gives $X=\Theta CC^{\top}\Theta-Q$. Thus, for all vertices $u,v$,
\begin{equation}\label{eq-Xtwo}
X_{uv}=(RR^{\top})_{uv}-P_{uv} \qquad\text{and}\qquad
X_{uv}=\theta_u\theta_v(CC^{\top})_{uv}-Q_{uv}.
\end{equation}

\medskip\noindent\textbf{Claim 4.} \emph{Let $v$ be any vertex. If $u,u'$ lie in the same row, then $X_{uv}=X_{u'v}$. If $u,u'$ lie in the same column, then $\theta_uX_{uv}=\theta_{u'}X_{u'v}$.}\\[2mm]
Let $u,u'$ lie in the same row $i$. Then $(RR^{\top})_{uv}=(RR^{\top})_{u'v}$, and $p_u=p_{u'}=\hat p_i$ by Claim~1, so $P_{uv}=P_{u'v}$. The first formula in \eqref{eq-Xtwo} gives $X_{uv}=X_{u'v}$.\\[-3mm]

Let $u,u'$ lie in the same column $j$. Multiplying the second formula in \eqref{eq-Xtwo} by $\theta_u$ gives $\theta_uX_{uv}=\theta_v(CC^{\top})_{uv}-\theta_uQ_{uv}$. Here
$(CC^{\top})_{uv}=(CC^{\top})_{u'v}$, and by Claim~1 $\theta_uQ_{uv}=\theta_u^2\langle\hat q_j,q_v\rangle=\langle\hat q_j,q_v\rangle$. The same holds for $u'$, so $\theta_uX_{uv}=\theta_{u'}X_{u'v}$. 

\medskip
Call two vertices \emph{neighbours} if they lie in the same row or in the same column.

\medskip\noindent\textbf{Claim 5.} \emph{Any two vertices $u,u'$ are joined by a sequence of vertices in which any two consecutive vertices are neighbours.}\\[2mm]
If $r(u)=r(u')$, there is nothing to prove. Otherwise, since $\mathcal{B}(W)$ is connected by Lemma~\ref{lem-xor}, it contains a path $r(u)=i_0,j_1,i_1,j_2,\ldots,j_m,i_m=r(u')$, in which rows and columns alternate. Each edge of this path is a nonempty square, so the squares
$(i_0,j_1),(i_1,j_1),(i_1,j_2),\ldots,(i_m,j_m)$ are nonempty. Consecutive squares in this list share a row or a column. Choose one vertex in each square, and add $u$ at the beginning and $u'$ at the end. This gives the required sequence. 

\medskip
Let $U$ be the set of vectors $y\in\mathbb{R}^{V}$ such that $y_u=y_{u'}$ whenever $u,u'$ lie in the same row, and $\theta_uy_u=\theta_{u'}y_{u'}$ whenever $u,u'$ lie in the same column.

\medskip\noindent\textbf{Claim 6.} \emph{$U$ is a subspace of $\mathbb{R}^V$ and $\dim U\leq1$.}\\[2mm]
The conditions defining $U$ are linear, so $U$ is a subspace. Fix a vertex $w$ and let $y\in U$ with $y_w=0$. If $y_u=0$ and $u'$ is a neighbour of $u$, then $y_{u'}=0$, where in the column case we use $\theta_{u'}=\pm1$. By Claim~5, every vertex is reached from $w$ through neighbours, so $y=0$. Hence the linear map $y\mapsto y_w$ is injective on $U$, and $\dim U\leq1$. 

\medskip\noindent\textbf{Claim 7.} \emph{There is a vector $s\in U$ with $s_v\in\{1,-1\}$ for every $v$ such that $X=e\,ss^{\top}$. In particular, $s_u=s_{u'}$ whenever $u,u'$ lie in the same row, and $\theta_us_u=\theta_{u'}s_{u'}$ whenever $u,u'$ lie in the same column.}\\[2mm]
By Claim~4, every column of $X$ lies in $U$. Hence the column space of $X$ is contained in $U$, and $\operatorname{rank}X\leq1$ by Claim~6. On the other hand, $X_{vv}=(RR^{\top})_{vv}-P_{vv}=1-d=e\neq0$ for every $v$, so $\operatorname{rank}X=1$. A
symmetric matrix of rank one has the form $X=\lambda yy^{\top}$, where $\lambda\neq0$ and $y$ is a nonzero vector spanning its column space. In particular $y\in U$. Comparing diagonal entries gives $\lambda y_v^2=e$ for every $v$, so all $|y_v|$ are equal and nonzero. Dividing $y$ by $|y_v|$, we obtain a vector $s\in U$ with entries $\pm1$ and $X=e\,ss^{\top}$. The last statement is the definition of $U$. 

\medskip\noindent\textbf{Claim 8.} \emph{Let $D=\operatorname{diag}(s_v)$. Then
$DXD=eJ$, $DRR^{\top}D=RR^{\top}$ and $D\,\Theta CC^{\top}\Theta\,D=CC^{\top}$.}\\[2mm] For every matrix $M$ we have $(DMD)_{uv}=s_us_vM_{uv}$. By Claim~7, $(DXD)_{uv}=e\,s_u^2s_v^2=e$. Next, $(DRR^{\top}D)_{uv}=s_us_v(RR^{\top})_{uv}$ is
nonzero only if $u,v$ lie in the same row, and then $s_us_v=1$ by Claim~7. Finally, $(D\Theta CC^{\top}\Theta D)_{uv}=(\theta_us_u)(\theta_vs_v)(CC^{\top})_{uv}$ is nonzero only if $u,v$ lie in the same column, and then $(\theta_us_u)(\theta_vs_v)=1$ by Claim~7.
\hfill$\diamond$

\medskip
We can now finish the proof. By Lemma~\ref{lem-parts}(4), Claim~3, Claim~8 and $P=RR^{\top}-X$,
\[
A(\Gamma^D)=DAD=RR^{\top}-CC^{\top},\qquad P(\Gamma^D)=DPD=RR^{\top}-eJ,
\]
and hence $Q(\Gamma^D)=P(\Gamma^D)-A(\Gamma^D)=CC^{\top}-eJ$. The switching in Claim~1 followed by $D$ is again a single switching of the original graph. Finally, $A(\Gamma^D)=A(\Gamma[W])$ by Lemma~\ref{lem-grid}(i).
\end{proof}


It remains to show that the matrix $W$ of Lemma~\ref{lem-xor} has the additive form required in Definition~\ref{def-rook}. We first bound the size of the board.

\begin{lemma}\label{lem-few}
Suppose $e<d$, and let $W$ be the $s\times t$ matrix of Lemma~\ref{lem-xor}. Then the board of $\Gamma[W]$ has at most $n/\alpha$ rows and at most $n/\alpha$ columns, that is,
\[
\max(s,t)\leq\frac{n}{\alpha}.
\]
\end{lemma}

\begin{proof}
Choose one vertex in each row, and let $x\in\mathbb{R}^{V}$ be the sum of the indicator vectors of these $s$ vertices. Then $R^{\top}x$ is the all-ones vector of length $s$, so
$x^{\top}RR^{\top}x=s$. Also $x^{\top}Jx=s^2$. By Proposition~\ref{prop-normal}, the matrix $RR^{\top}-eJ=P(\Gamma^D)$ is positive semidefinite. Hence
\[
0\leq x^{\top}\bigl(RR^{\top}-eJ\bigr)x=s-es^2 ,
\]
which gives $s\leq1/e=n/\alpha$. The same argument with $CC^{\top}-eJ=Q(\Gamma^D)$ and one vertex in each column gives $t\leq n/\alpha$.
\end{proof}


\begin{lemma}\label{lem-additive}
Suppose $e<d$, let $W$ be the $s\times t$ matrix of Lemma~\ref{lem-xor}, and put
$N=\max(s,t)$. Then, after adding zero rows or zero columns, $W$ becomes an $N\times N$ matrix with entries $w_{ij}=a_i+b_j$ for some $a,b\in\mathbb{Z}_{\geq0}^{N}$. Consequently, $\Gamma[W]=\Rook_N(a,b)$.
\end{lemma}

\begin{proof}
By Lemma~\ref{lem-xor}, $\Gamma[W]$ has the same underlying graph as $\Gamma$. Hence
$\alpha(\Gamma[W])=\alpha$ and $\sum_{i,j}w_{ij}=n$.\\[2.5mm]
\noindent\textbf{Case: $s\leq t$.}\\[1.5mm]
By Lemma~\ref{lem-rook}(b), $\alpha\geq n/t$. By Lemma~\ref{lem-few},
$t\leq n/\alpha$, that is, $\alpha\leq n/t$. So $\alpha=n/t$, and Lemma~\ref{lem-rook}(c),(d) gives two possibilities. If $s<t$, then $w_{ij}=a_i$ for all $i,j$. Adding $t-s$ zero rows does not change $\Gamma[W]$ by Lemma~\ref{lem-grid}(iv), and the result is
$\Rook_t\bigl((a_1,\ldots,a_s,0,\ldots,0),\mathbf{0}\bigr)$. If $s=t$, then $w_{ij}=a_i+b_j$ for some integers $a_i,b_j$. By Remark~\ref{rem-nonneg} they may
be chosen nonnegative, and so $\Gamma[W]=\Rook_s(a,b)$.\\[2mm]
\noindent\textbf{Case: $s>t$.}\\[1.5mm]
By Lemma~\ref{lem-grid}(iii), $\Gamma[W^{\top}]=-\Gamma[W]$ has the same
underlying graph as $\Gamma$, and its board has $t$ rows and $s$ columns. Applying the previous case to $W^{\top}$ gives $w_{ij}=b_j$ for all $i,j$. Adding $s-t$ zero columns, we obtain $\Gamma[W]=\Rook_s\bigl(\mathbf{0},(b_1,\ldots,b_t,0,\ldots,0)\bigr)$.
\end{proof}

We can now complete the proof of Theorem~\ref{thm-nec}.

\begin{proof}[\textbf{Proof of Theorem~\ref{thm-nec}}]
If $d=e$, then $\Gamma\sim\Rook_2\bigl((a,0),(0,0)\bigr)$ by Proposition~\ref{prop-boundary}. If $d>e$, then $\Gamma\sim\Gamma[W]$ by Proposition~\ref{prop-normal}, and
$\Gamma[W]=\Rook_N(a,b)$ by Lemma~\ref{lem-additive}. Finally, $N\geq2$, because $\Gamma$ has an edge while $\Rook_1(a,b)$ has none.
\end{proof}

\section{Proof of the main theorem}\label{sec-proof}

\begin{proof}[Proof of Theorem~\ref{thm-main}]
The inequality $\mathcal{E}(\Gamma)\geq2\tau(\Gamma)$ is Theorem~\ref{thm-ineq}. It remains to characterize equality.\\[-3mm]

Suppose first that every connected component of $\Gamma$ is switching equivalent to an additive rook graph. By Theorem~\ref{thm-suff}, every component attains equality, and therefore $\mathcal{E}(\Gamma)=2\tau(\Gamma)$ by Corollary~\ref{cor-comp}.\\[-3mm]

Conversely, suppose that $\mathcal{E}(\Gamma)=2\tau(\Gamma)$. By Corollary~\ref{cor-comp}, every connected component of $\Gamma$ attains equality. A component with one vertex is $K_1=\Rook_1\bigl((1),(0)\bigr)$. A component with at least two vertices is switching equivalent to an additive rook graph by Theorem~\ref{thm-nec}. Hence every component of $\Gamma$ is switching equivalent to an additive rook graph.
\end{proof}

\section{Consequences and limitations}\label{sec-consequences}

We call a signed graph $\Gamma$ \emph{extremal} if $\mathcal{E}(\Gamma)=2\tau(\Gamma)$. In this section we derive some consequences of Theorem~\ref{thm-main}. We compute the basic parameters of a connected extremal signed graph, decide when it is balanced or antibalanced, and recover the characterization of extremal unsigned graphs.

\subsection{Unsigned graphs}\label{subsec-unsigned}

\begin{corollary}\label{cor-inv}
Let $\Gamma$ be a connected extremal signed graph with $\Gamma\sim\Rook_N(a,b)$. Then
\[
n=N\ell,\qquad \alpha(\Gamma)=\ell,\qquad \mathcal{E}(\Gamma)=2(N-1)\ell .
\]
Moreover, every vertex $v$ has the same vertex energy $|A(\Gamma)|_{vv}=\frac{2(N-1)}{N}$.
\end{corollary}

\begin{proof}
Switching changes neither $n$, $\alpha$ nor the energy. So the first three equalities follow from Corollary~\ref{cor-rookinv} and Theorem~\ref{thm-suff}. If $n\geq2$, then
Proposition~\ref{prop-rigid}(a) gives $|A(\Gamma)|_{vv}=2\tau(\Gamma)/n=2(N-1)/N$. If $n=1$, then $\Gamma=K_1$, $N=1$, and both sides are $0$.
\end{proof}


For an additive rook graph we say that a row or a column of the board is \emph{nonempty} if it contains at least one vertex.

\begin{theorem}\label{thm-balance}
Let $\Gamma$ be a connected extremal signed graph with $\Gamma\sim\Rook_N(a,b)$.
\begin{enumerate}[label=(\alph*)]
  \item The following are equivalent:
  \begin{enumerate}[label=(\roman*)]
    \item $\Gamma$ is balanced;
    \item the board of $\Rook_N(a,b)$ has at most two nonempty rows;
    \item $\Gamma$ is switching equivalent to $K_1$, to $K_{x,\ldots,x}$ with $r\geq2$ parts, or to $H_r(x,y)$ with $r\geq3$.
  \end{enumerate}
  \item The following are equivalent:
  \begin{enumerate}[label=(\roman*)]
    \item $\Gamma$ is antibalanced;
    \item the board of $\Rook_N(a,b)$ has at most two nonempty columns;
    \item $\Gamma$ is switching equivalent to $K_1$, to $-K_{x,\ldots,x}$ with $r\geq2$ parts, or to $-H_r(x,y)$ with $r\geq3$.
  \end{enumerate}
  \item If the board has at least three nonempty rows and at least three nonempty columns, then $\Gamma$ is neither balanced nor antibalanced.
\end{enumerate}
\end{theorem}

\begin{proof}
Switching does not change balance or antibalance. Hence we may assume $\Gamma=\Rook_N(a,b)$.

\noindent(a) We first show that (ii) implies (i). If only one row is nonempty, every edge joins two vertices of that row and is positive. If exactly two rows are nonempty, switch at all vertices of one of them. A row edge has both ends switched or neither, so it stays positive. A column edge joins the two nonempty rows, so exactly one of its ends is switched, and it becomes positive. In both cases $\Gamma$ is switching equivalent to an all-positive graph, and so it is balanced.\\[-3mm]

Next we show that (i) implies (ii). Suppose rows $i_1,i_2,i_3$ are nonempty. We claim that some column meets all three of them in nonempty squares. If $b_j>0$ for some $j$, then every square of column $j$ has $a_i+b_j>0$ vertices. If $b=\mathbf{0}$, then $a_{i_1},a_{i_2},a_{i_3}>0$, and every square in these three rows is nonempty. Taking one vertex from each of the three squares, we obtain three vertices in the same column and in different rows. They are pairwise joined by negative edges, so they form a negative triangle. Hence $\Gamma$ is not balanced.\\[-3mm]

It remains to show that (ii) is equivalent to (iii). The graphs listed in (iii) are all-positive, so (iii) implies (i), and hence (ii). Conversely, assume (ii).\\[-3mm]

Suppose first that $b_j>0$ for some $j$. Then every row is nonempty, so $N\leq2$. If $N=1$, the graph has no edges, and since it is connected, $\Gamma=K_1$. If $N=2$, switch at the vertices of row $2$. As above, every edge becomes positive. Row edges join $V_{11}$ with $V_{12}$ and $V_{21}$ with $V_{22}$, and column edges join $V_{11}$ with $V_{21}$ and $V_{12}$ with $V_{22}$. Hence the switched graph is the complete bipartite graph with parts $V_{11}\cup V_{22}$ and $V_{12}\cup V_{21}$. Both parts have $\ell$ vertices, so $\Gamma\sim K_{\ell,\ell}$.\\[-3mm]

Suppose now that $b=\mathbf{0}$, so that $w_{ij}=a_i$ and the nonempty rows are those with
$a_i>0$. If only row $i$ is nonempty, then $\Gamma$ is the all-positive $K_{a_i,\ldots,a_i}$ with $N$ parts, and connectivity gives either $N\geq2$ or $\Gamma=K_1$. If exactly the rows $i$ and $k$ are nonempty, then switching at the vertices of row $k$ gives the unsigned graph $H_N(a_i,a_k)$, as in Example~\ref{ex-members}. For $N=2$ this graph is $K_{a_i+a_k,\,a_i+a_k}$, and for $N\geq3$ it is $H_N(a_i,a_k)$ with $N\geq3$.\\[-2mm]

\noindent(b) By Lemma~\ref{lem-grid}(iii), $-\Gamma=\Rook_N(b,a)$, and the rows of its board are the columns of the board of $\Gamma$. Moreover, $-\Gamma$ is connected and extremal, since negation changes neither the energy nor the underlying graph. Now apply (a) to $-\Gamma$.\\[-3mm]

\noindent(c) By the proof of (a), at least three nonempty rows give a negative triangle, so $\Gamma$ is not balanced. Similarly, at least three nonempty columns give three vertices in one row and in different columns, which form a positive triangle. This triangle is negative in $-\Gamma$, so $\Gamma$ is not antibalanced.
\end{proof}

\begin{remark}
Cases (a) and (b) overlap exactly when the board has at most two nonempty rows and at most two nonempty columns. Up to switching, this happens only for $K_1$ and $K_{x,x}$. These are the connected extremal signed graphs that are both balanced and antibalanced.
\end{remark}

Table~\ref{tab-balance} summarizes Theorem~\ref{thm-balance}.

\begin{table}[ht]
\centering
\begin{tabular}{@{}>{\raggedright\arraybackslash}p{3.2cm}>{\raggedright\arraybackslash}p{4.4cm}>{\raggedright\arraybackslash}p{6.4cm}@{}}
\toprule
Connected extremal $\Gamma$ & Board of $\Rook_N(a,b)$ & Graphs up to switching\\
\midrule
balanced & at most two nonempty rows & $K_1$, $K_{x,\ldots,x}$, $H_r(x,y)$\\
antibalanced & at most two nonempty columns & $K_1$, $-K_{x,\ldots,x}$, $-H_r(x,y)$\\
neither & at least three nonempty rows and at least three nonempty columns & for example
$\Rook_3\bigl((1,1,1),\mathbf{0}\bigr)$ and $\Rook_3\bigl((1,0,0),(0,1,0)\bigr)$\\
\bottomrule
\end{tabular}
\caption{Balance of the connected extremal signed graphs.}
\label{tab-balance}
\end{table}

Finally, we recover the characterization of extremal unsigned graphs due to Mojallal~\cite{Mojallal}.

\begin{corollary}\label{cor-unsigned}
An unsigned graph $G$ satisfies $\mathcal{E}(G)=2\tau(G)$ if and only if every connected
component of $G$ is $K_1$, a complete multipartite graph $K_{x,\ldots,x}$ with $r\geq2$ parts, or a graph $H_r(x,y)$ with $r\geq3$.
\end{corollary}

\begin{proof}
Regard $G$ as the all-positive signed graph. Suppose $\mathcal{E}(G)=2\tau(G)$. By Corollary~\ref{cor-comp}, every connected component is extremal, and it is balanced because all its edges are positive. By Theorem~\ref{thm-main}, it is switching equivalent to an additive rook graph. By Theorem~\ref{thm-balance}(a), it is therefore switching equivalent to one of the listed graphs. Switching does not change the underlying graph, so the component is isomorphic to that graph. Conversely, each listed graph is switching equivalent to an additive rook graph by Example~\ref{ex-members}, so it is extremal by Theorem~\ref{thm-main}. Corollary~\ref{cor-comp} then gives $\mathcal{E}(G)=2\tau(G)$.
\end{proof}


Corollary~\ref{cor-unsigned} shows that the extremal unsigned graphs are exactly the balanced connected extremal signed graphs, up to switching. The family of extremal signed graphs is strictly larger. In particular, a signed graph can be extremal even when its underlying graph is not extremal as an unsigned graph.

\begin{example}\label{ex-K3K3}
Let $\Gamma=\Rook_3\bigl((1,1,1),\mathbf{0}\bigr)$. This is the graph $K_3\,\square\,K_3$ with positive rows and negative columns. By Corollary~\ref{cor-rookinv} we have $\tau(\Gamma)=6$. The spectrum of $\Gamma$ is $\{3^{[2]},0^{[5]},(-3)^{[2]}\}$, so $\mathcal{E}(\Gamma)=12=2\tau(\Gamma)$. On the other hand, the underlying graph $K_3\,\square\,K_3$ has spectrum $\{4^{[1]},1^{[4]},(-2)^{[4]}\}$ and energy $16=\mathcal{E}(K_3\,\square\,K_3)>2\tau(K_3\,\square\,K_3)=12$.
\end{example}


\subsection{Failure of the bound for $\mathbb{T}$-gain graphs}\label{subsec-gain}

A \emph{complex unit gain graph} (or \emph{$\mathbb{T}$-gain graph}) $\Psi=(G,\psi)$ assigns to each oriented edge $\overrightarrow{uv}$ a complex number $\psi(\overrightarrow{uv})$ of modulus one, with $\psi(\overrightarrow{vu})=\overline{\psi(\overrightarrow{uv})}$ \cite{Reff, MehatariKannanSamanta, SamantaKannanSpectrum}. Its adjacency matrix $A(\Psi)$ is the Hermitian matrix with $A(\Psi)_{uv}=\psi(\overrightarrow{uv})$ if $u\sim v$ and $A(\Psi)_{uv}=0$ otherwise. The Hermitian adjacency matrix of a mixed graph \cite{GuoMohar,LiuLi} is the special case in which every gain is $1$ (undirected edge) or $\pm i$ (arc). Signed graphs are exactly the $\mathbb{T}$-gain graphs with all gains in $\{1,-1\}$. The next example shows that Theorem~\ref{thm-main} does not extend beyond this case.

\begin{example}\label{ex-gain}
Let $\Psi=(K_3,\psi)$ be the $\mathbb{T}$-gain graph on the vertices $v_1,v_2,v_3$ with
$\psi(\overrightarrow{v_1v_2})=\psi(\overrightarrow{v_2v_3})=\psi(\overrightarrow{v_3v_1})=i$. Then
\[
A(\Psi)=\begin{bmatrix}
0 & i & -i\\
-i & 0 & i\\
i & -i & 0
\end{bmatrix},
\]
which is also the Hermitian adjacency matrix of the directed triangle $v_1\to v_2\to v_3\to v_1$. Its characteristic polynomial is $\lambda^3-3\lambda$, so its eigenvalues are $-\sqrt3$, $0$ and $\sqrt3$. Since $\tau(K_3)=2$, we get
\[
\mathcal{E}(\Psi)=2\sqrt3<4=2\tau(K_3).
\]
\end{example}


It can be seen that the signed graphs form the largest class of $\mathbb{T}$-gain graphs to which Theorem~\ref{thm-main} applies in its present form. For $\mathbb{T}$-gain graphs only weaker bounds are known, such as $\mathcal{E}(\Psi)\geq2\tau(G)-2c(G)$, where $c(G)$ is the number of odd cycles of $G$~\cite{GainWM}.

\vspace*{3mm}

\noindent
\textbf{Acknowledgment:}
Aniruddha Samanta thanks the Birla Institute of Technology, Mesra, for financial support under the Institute Seed Money Scheme (Ref.\ No.\ DRIE/SMS/DRIE-10989/2026-27/1977). K. C. Das is supported by the National Research Foundation funded by the Korean government (Grant No. RS-2026-25475577).

\vspace*{3mm}

\noindent
\textbf{Data Availability:} Not applicable.

\vspace*{3mm}

\noindent
\textbf{Declaration of competing interest:} The authors declare that they have no competing financial interests or personal relationships that could have influenced this work.

 \vspace*{3mm}


\end{document}